\PassOptionsToPackage{square,comma,numbers,sort&compress}{natbib} 
\documentclass{article}

\usepackage[final]{neurips_2021}
\usepackage{bbding} 

\usepackage[utf8]{inputenc} 
\usepackage[T1]{fontenc}    
\usepackage{url}            
\usepackage{booktabs}       
\usepackage{amsfonts}       
\usepackage{nicefrac}       
\usepackage{microtype}      
\usepackage{wrapfig}

\usepackage{amssymb}

\usepackage{amsmath}
\usepackage{amsfonts}
\usepackage{algorithm}
\usepackage{algorithmic}

\newtheorem{theorem}{Theorem}
\newtheorem{lemma}{Lemma}

\newtheorem{definition}{Definition}
\newtheorem{note}{Note}

\newtheorem{conjecture}{Conjecture}
\newtheorem{proof}{Proof}

\usepackage{times}
\usepackage{epsfig}
\usepackage{graphicx}
\usepackage{amsmath}
\usepackage{amssymb}

\usepackage{amsfonts}
\usepackage{algorithm}
\usepackage{algorithmic}
\usepackage{color}

\usepackage{lineno,hyperref}
\usepackage{amsfonts}
\usepackage{multirow}
\usepackage{bm}

\usepackage{diagbox}
\usepackage{multirow}

\usepackage{times}
\usepackage{soul}
\usepackage{url}
\usepackage{algorithm}
\usepackage{algorithmic}
\usepackage{subfigure}
\usepackage{lineno,hyperref}
\usepackage{multirow}
\usepackage{bm}
\modulolinenumbers[1]
\usepackage{rotating}
\usepackage{booktabs}
\usepackage{makecell}

\usepackage{wrapfig,lipsum,booktabs}

\usepackage{amsfonts}       
\usepackage{nicefrac}       
\usepackage{microtype}      
\usepackage{xcolor}         

\title{Special Stable Matrices and Their Non-square Counterpart}

\author{ Steven W. Su $^1$ $^*$
\AND
\thanks{$^{*}$ The co-responding author.}
\thanks{$^{1}$ College of Artificial Intelligence and Big data for Medical Sciences, Shandong First Medical University
Shandong Academy of Medical Sciences, P. R. China.}
\thanks{$^{2}$ University of Technology, Sydney, Australia.}  
}

\begin{document}

\maketitle

\begin{abstract}
In this note, we discuss the extension of several important stable square matrices, e.g., D-stable matrices, diagonal dominance matrices, Volterra-Lyapunov stable matrices, to their corresponding non-square matrices. The extension is motivated by some distributed control-related problems, such as decentralized unconditional stability and decentralized integral controllability for non-square processes. We will provide the connections of conditions between these special square matrices and their associated non-square counterparts. Some conjectures for these special matrices are proposed for future research.    
\end{abstract}

\section{INTRODUCTION}

In various applications, the stability of a system under different structural variations and constraints is a major concern in relevant research areas. For example, in the linearization of diffusion models in biological systems at constant equilibrium, strong stable matrices are proposed \cite{CROSS1978253, hadeler2006nonlinear}. The field of system control and its applications, including decentralized stability \cite{sun2023gallery, Wang:-On, hadeler2006nonlinear, johnson1974sufficient, goh1977global, iggidr2023limits, barker1978positive}, has motivated research into the stability of various matrices, such as P-matrices, D-stable matrices, and Volterra-Lyapunov stable matrices, which are also significant in other areas like economics.

In the process control industry, decentralized or distributed control is often preferred over centralized control due to its simplicity and fault tolerance. As a result, the configuration of decentralized control has been extensively investigated. Topics such as Decentralized Unconditional Stability and Decentralized Controllability have been well explored \cite{anderson1981algebraic, Anderson:-Alge, Sko:-ariab, grosdidier1986interaction, Su:-Analy}.

In some cases, particularly in applications involving non-square systems, the stability of some special real matrices associated with these systems must be addressed \cite{zhang2017multiloop, WHEATON20171, zhiteckii2022robust, sujatha2022control, steentjes2022data}. This need has led to the extension of stability concepts to non-square matrices. This note focuses exclusively on real matrices, with the potential for extension to complex matrices via frequency-based techniques.

This study investigates the connections between the stability of square matrices and their non-square counterparts and presents conjectures for future research.

\section{DECENTRALIZED STABILIZATION FOR NON-SQUARE SYSTEMS AND MATRICES}

The motivation for the extension of the decentralized stability results from square matrices to non-square matrices is due to the exploration of the decentralized unconditional stability (DUS)\cite{Sko:-ariab} \cite{grosdidier1986interaction} \cite{Su:-Analy} of non-square processes. A sufficient DUS condition is the D-stability alike requirement for the system steady state gain matrix, which is a non-square real matrix for the investigation of non-square processes. We present the D-stability-like sufficient condition of non-square matrices here for the decentralized unconditional stability. As we mainly concern with the investigation of real non-square matrices in this note, we will not provide the details about how to use this condition to ensure DUS of non-square processes. We mainly focus on the connection of DUS conditions of the steady state transfer functions between non-square and its related square systems. Suppose the readers are interested in the proof of this sufficient DUS condition. In that case, we recommend reading the existing relevant references \cite{Sko:-ariab} \cite{grosdidier1986interaction} \cite{Su:-Analy} about the proof of the sufficient DUS condition for square systems via singular perturbation analysis \cite{Kokotovic:-Sing}. 

We present the sufficient DUS condition as follows.

\begin{theorem} \label{theorem1}
For a proper and asymptotically stable non-square system $H(s) \in \mathbb{C}^{m \times n}$, we define its steady state transfer matrix as $A=H(0) \in \mathbb{R}^{n \times m}$. If there exists a 
 block diagonal non-square matrix $K \in \mathbb{C}^{n \times m}$ such that for all non-negative diagonal matrices $E \in  \mathbb{R}^{n \times n}$, and $j \in \mathcal{M}$ \cite{campo1994achievable},
\begin{equation}\label{suf_con}
Re\{ \sigma_i ( [A E K]_j ) \} >0,
\end{equation}
where $\mathcal{M}$ is the index set consisting of $k$ tuples of integers in the range $1,\cdots, m$, and $\text{Re} \{ \sigma_i (M) \}$ represents the real part of the $i$-th eigenvalue of matrix $M$, 

then the nonsquare system can be controlled by a distributed integral controller and achieve Decentralized Unconditional Stability.
\end{theorem}

Condition (\ref{suf_con}) in Theorem \ref{theorem1} describes the characteristic of a non-square D-stable alike matrix. Using Theorem \ref{theorem1} as the starting point, we investigate the D-stable conditions between non-square and square matrices now. To formalize our discussion, we define the structure of the $K$ matrix. According to this structure, we will define the squared matrices for the non-square matrices $A$ and $E$.    

First, we assume the structure of the block diagonal $K$ matrix as follows:
\begin{equation}\label{integral_matrix}
K=
\begin{bmatrix}
    k_{1,1} & \cdots & k_{1, p_1} & 0 & \cdots & 0 & \cdots & 0 & \cdots & 0 \\
    0 & \cdots & 0 & k_{2, 1} & \cdots & k_{2, p_2} & \cdots & 0 & \cdots & 0\\
    \vdots & \vdots & \ddots & \vdots & \vdots & \ddots & \vdots & \vdots & \ddots & \vdots \\
    0 & \cdots & 0 & 0 & \dots & 0 & \cdots & k_{m, 1} & \cdots & k_{m, p_m}
\end{bmatrix}^T
\end{equation}
where $\sum_{i=1}^{m} p_i = n$, and all the elements $k_{i,j}$ are non-negative.

Considering the non-negative diagonal matrix $ E= diag\{ [\varepsilon _{1,1}, \cdots,
  \varepsilon_{m,p_m}] \} \,\,\, \,\, $ with $\varepsilon _{i} \ge 0 $, we have 
\begin{equation}\label{integral_matrix_E}
\begin{aligned}
\bar K& = E K=\\
&\begin{bmatrix}
    k_{1,1} \varepsilon _{1,1} & \cdots & k_{1, p_1} \varepsilon _{1, p_1} & 0 & \cdots & 0 & \cdots & 0 & \cdots & 0 \\
    0 & \cdots & 0 & k_{2, 1} \varepsilon _{2, 1} & \cdots & k_{2, p_2} \varepsilon _{2, p_2} & \cdots & 0 & \cdots & 0\\
    \vdots & \vdots & \ddots & \vdots & \vdots & \ddots & \vdots & \vdots & \ddots & \vdots \\
    0 & \cdots & 0 & 0 & \dots & 0 & \cdots & k_{m, 1} \varepsilon _{m,1} & \cdots & k_{m, p_m} \varepsilon _{m, p_m}
\end{bmatrix}^T
\end{aligned}
\end{equation}  

According to the structure of $K$, we redefine the sub-indices of the $n$ columns of $A$ as follows: 
$$A=[\boldsymbol{a}_{1,1}, \cdots, \boldsymbol{a}_{1,p_1}, \boldsymbol{a}_{2,1}, \cdots, \boldsymbol{a}_{2,p_2}, \cdots, \cdots, \boldsymbol{a}_{m,1} \cdots,\boldsymbol{a}_{m,p_m}],$$ 
then $AEK$ as shown in (\ref{suf_con}) can be expressed as 
\begin{equation}\label{eq_li_comb}
AEK = [ \sum_{i=1}^{p_1} \varepsilon_{1,i} k_{1,i} \boldsymbol{a}_{1,i}, \sum_{i=1}^{p_2} \varepsilon_{2, i} k_{2,i} \boldsymbol{a}_{2,i}, \cdots, \sum_{i=1}^{p_m} \varepsilon_{m,i} k_{m,i} \boldsymbol{a}_{m,i}].
\end{equation}

Now, based on the structure (\ref{integral_matrix}), we define the Squared Matrices for the matrices $A$ and $K$ as follows.





\begin{definition} (Squared Matrices) \label{squaredmatrix}
Consider a non-square matrix $A \in \mathbb{R}^{m \times n}$ and its associated block diagonal matrix $K \in \mathbb{R}^{m \times n}$, as defined in Equation (\ref{integral_matrix}). We describe a procedure to construct squared matrices from $A$ and $K$ regarding the values of $\varepsilon_{i,j}$.

First, for the case where the dimension of the squared matrix is $m$, i.e., $\forall i \in \{1, 2, \cdots, m\}$, we select only one column from each $\boldsymbol{a}_{i, j}$, where $j \in \{1, 2, \cdots, p_i\}$. This corresponds to the scenario where, for a given $i$, exactly one of the corresponding $\varepsilon_{i,j} \ne 0$, $j \in \{1, 2, \cdots, p_i\}$. We can construct $N = \prod_{i=1}^{m} p_i$ such squared matrices for both $A$ and $K$, denoted as $[A]^m_{s_i}$ and $[K]^m_{s_i}$, respectively, for $i \in \{1, 2, \cdots, N\}$.

Second, when the dimension of the squared matrices is less than $m$, corresponding to the scenario where, for one or more $i \in \{1, 2, \cdots, m\}$, $\varepsilon_{i,j} = 0$ for all $j \in \{1, 2, \cdots, p_i\}$, we can construct squared matrices of dimension $k < m$. These are denoted as $[A]^k_{s_i}$ and $[K]^k_{s_i}$.

These squared matrices are referred to as the square counterparts of the non-square matrices $A$ and $K$.
\end{definition}

In Definition \ref{squaredmatrix},  regarding the original decentralized stabilization problem, the reduction in the dimension of matrices occurs in cases where all the control inputs for certain columns have zero gain. In other words, for some index $i$, if $\varepsilon_{i,j} = 0$ for all $j \in {1, 2, \cdots, p_i}$, it corresponds to the removal of a control input from service, as discussed in \cite{campo1994achievable}.

Based on the definition of the squared matrices (Definition \ref{squaredmatrix}), we present the connection between square and non-square matrices. 

First, we introduce two well-known stable square matrices.

\begin{definition} \cite{Cross:-Thre}
The matrix $A=(a_{ij}) \in \mathbb{R}^{n \times n}$ is said to be

(1) D-stable if $DA$ is stable for all diagonal matrix $D>0$;

(2) Volterra-Lyapunov stable if there exists a diagonal matrix $D >0$ for which $AD + DA^T > 0$.
\end{definition}

\begin{definition} \cite{KUSHEL2021204}
A square matrix $A$ is said to be strictly column diagonally dominant if for each row $i$, the absolute value of the diagonal element is greater than the sum of the absolute values of the other elements in the same column:
\[
|a_{ii}| \geq \sum_{j > i} |a_{ij}| \quad \text{for all } i.
\]    
\end{definition}

\begin{definition}
A matrix $A$ is called strictly column diagonally balanced dominant if there exists a diagonal matrix $D$ such that $|AD| + |D A^T|$ is strictly column diagonally dominant. Furthermore, if all the diagonal elements of $AD + D A^T $ are positive then the matrix $A$ is called positive diagonally balanced dominant. 
\end{definition}

\begin{lemma} (\cite{Cross:-Thre}) \label{lemma2} 
If a matrix $A=(a_{ij}) \in \mathbb{R}^{n \times n}$ is Volterra-Lyapunov stable, then, for any diagonal positive matrix $D > 0$ and $i \in \mathcal{M}$, where $\mathcal{M}$ is the index set consisting of $k$ tuples of integers in the
range $1,\cdots, m$,
$$Re\{ \sigma [AD]_i \} >0,$$
which implies that $[A]_i$, for $i \in \mathcal{M}$,  are all D-stable matrices.
\end{lemma}

\begin{definition} \label{def_NSQ_lyp_ind}
For a real nonsquare matrix $A \in \mathbb{R}^{m \times n}$, if one can find an individual diagonal positive matrix  $D_i > 0$, so that all its $m$-order squared matrices (see Definition \ref{squaredmatrix}) $[A]^m_{s_i} \in \mathbb{R}^{m \times m}$  satisfy the following LMIs,
$$[A]^m_{s_i} D_i + D_i ([A]^m_{s_i})^T  > 0, $$
then, we call the non-square matrix $A$ an individual Volterra-Lyapunov stable matrix.
\end{definition}

\begin{definition} \label{def_NSQ_lyp}
For a real nonsquare matrix $A \in \mathbb{R}^{m \times n}$, if one can find a diagonal positive matrix  $D > 0$, so that all its $m$-order squared matrices (see Definition \ref{squaredmatrix}) $[A]^m_{s_i} \in \mathbb{R}^{m \times m}$  satisfy the following LMIs,
$$[A]^m_{s_i} D + D ([A]^m_{s_i})^T  > 0, $$
then, we call the non-square matrix $A$ a simultaneous Volterra-Lyapunov stable matrix.
\end{definition}

\begin{definition} \label{def_CDD_lyp}
For a real nonsquare matrix $A \in \mathbb{R}^{m \times n}$, if one can find a diagonal positive matrix  $D > 0$, so that 
$$ [A]^m_{s_i} D + D ([A]^m_{s_i})^T $$
are all strictly column diagonally dominant and all their diagonal elements are positive ( where $[A]^m_{s_i}$ is the $m$-order squared matrices (see Definition \ref{squaredmatrix}) of $A$), then, we call the non-square matrix $A$ simultaneous positive diagonally balanced dominant.
\end{definition}

\begin{note}
It is easy to see that a non-square matrix $A$ is simultaneous Volterra-Lyapunov stable if it is simultaneous positive diagonally balanced dominant.
\end{note}

\begin{lemma} \label{pailiezuhe_Lemma}
Consider $m$ groups of cards, where each group contains $p_i$ distinct cards. A player can place a bet on each combination of $m$ cards, one card from each group. The payoff for a specific combination is denoted as $\gamma(\zeta_1, \zeta_2, \cdots, \zeta_m) > 0$. We record the payoff for each card in a constructed combination as $\frac{\gamma(\zeta_1, \zeta_2, \cdots, \zeta_m)}{m}$. Suppose the player desires a specific ratio $\kappa_{kj} > 0$ between the $j_{th}$ card of group $k$ and the first card of that group.

If the set of payoffs $\gamma(\zeta_1, \zeta_2, \cdots, \zeta_m)$ satisfies the following conditions, then the player can implement their desired strategy:

\begin{equation}\label{sub_sssa}
\begin{cases}
\begin{aligned}
&\gamma_{(1, 1, \cdots, 1)} = 1 \\
&\gamma_{( \zeta_1, \zeta_2, \cdots, \zeta_{(i-1)}, j, \zeta_{(i+1)}, \cdots, \zeta_{m})} = \gamma_{( \zeta_1, \zeta_2, \cdots, \zeta_{(i-1)}, 1, \zeta_{(i+1)}, \cdots, \zeta_{m})} \kappa_{ij},
\end{aligned}
\end{cases}
\end{equation}
\end{lemma}

\begin{proof}
We will prove the desired strategy can be achieved by constructing a set of payoffs that satisfies the conditions in equation \eqref{sub_sssa}.

1. \textbf{Normalization Condition:} \\
First, according to Condition (\ref{sub_sssa}), we easily see that all $\gamma(\zeta_1, \zeta_2, \cdots, \zeta_m)$ can be set as a positive real value.
The condition $\gamma_{(1, 1, \cdots, 1)} = 1$ ensures that the payoff for choosing one card from each group, where all cards are the first cards of their respective groups, is normalized to 1. This establishes a baseline.

2. \textbf{Ratio Preservation Condition:} \\
According to Condition (\ref{sub_sssa}), we easily see that all $\gamma(\zeta_1, \zeta_2, \cdots, \zeta_m)$ can be set as a positive real value.
Then, the payoff of the first card in group $k$ is
$$
\frac{1}{m}(\sum_{\zeta_1=1}^{p_1} \cdots \sum_{\zeta_{k-1}}^{p_{k-1}} \sum_{\zeta_{k+1}}^{p_{k+1}} \cdots \cdots \sum_{\zeta_m}^{p_m} \gamma_{( \zeta_1, \cdots, \zeta_{k-1}, 1, \zeta_{k+1}, \cdots, \zeta_m )}).
$$
The payoff of the $j_{th}$ card in group $k$ is
\begin{equation}\label{eq_meili}
\begin{aligned}
&\frac{1}{m}(\sum_{\zeta_1=1}^{p_1} \cdots \sum_{\zeta_{k-1}}^{p_{k-1}} \sum_{\zeta_{k+1}}^{p_{k+1}} \cdots \cdots \sum_{\zeta_m}^{p_m} \gamma_{( \zeta_1, \cdots, \zeta_{k-1}, j, \zeta_{k+1}, \cdots, \zeta_m )})\\
=&\frac{1}{m}(\sum_{\zeta_1=1}^{p_1} \cdots \sum_{\zeta_{k-1}}^{p_{k-1}} \sum_{\zeta_{k+1}}^{p_{k+1}} \cdots \cdots \sum_{\zeta_m}^{p_m} \kappa_{k j} \gamma_{( \zeta_1, \cdots, \zeta_{k-1}, 1, \zeta_{k+1}, \cdots, \zeta_m )}).\\
\end{aligned}
\end{equation}
Thus, the second equation of Condition (\ref{eq_meili}) ensures that the payoff for choosing the $j_{th}$ card of group $k$ in combination with other cards is directly proportional to the payoff for choosing the first card of that group, with the ratio being $\kappa_{kj}$.

\end{proof}


\begin{lemma} \label{simul_lemma22a} 
If a real nonsquare matrix $A \in \mathbb{R}^{m \times n}$ is a simultaneous Volterra-Lyapunov stable matrix (see Definition \ref{def_NSQ_lyp}),
then, there exists a 
 block diagonal non-square matrix $K \in \mathbb{C}^{n \times m}$ such that for all non-negative diagonal matrices $E \in  \mathbb{R}^{n \times n}$, and $j \in \mathcal{M}$, where $\mathcal{M}$ is the index set consisting of $k$ tuples of integers in the
range $1,\cdots, m$,
\begin{equation}\label{suf_con_1}
Re\{ \sigma_i ( [A E K]_j ) \} >0,
\end{equation}
where $\text{Re} \{ \sigma_i (M) \}$ represents the real part of the $i$-th eigenvalue of matrix $M$.
\end{lemma}

\begin{proof}
Considering that $A$ is a simultaneous Volterra-Lyapunov stable matrix, the following inequality should be true:
\begin{equation}\label{eq_mdl}
[A]^m_{s_i} D + D ([A]^m_{s_i})^T  > 0. 
\end{equation}
Assume the total number of the $m_{th}$-order squared matrices $[A]^m_{s_i}$ is $N$, i.e., $i \in \{1,2, \cdots, N\}$. Based on inequality (\ref{eq_mdl}), we can conclude, for any $\gamma_i>0$, $i \in \{1, \cdots, N\}$, the following inequality holds
\begin{equation} \label{al_1}
\begin{aligned}
(\gamma_1 [A]^m_{s_1} + \gamma_2 [A]^m_{s_2} \cdots + \gamma_N [A]^m_{s_N})D & +\\
 D (\gamma_1 [A]^m_{s_1} + \gamma_2 [A]^m_{s_2} \cdots + \gamma_N [A]^m_{s_N})^T & >0
\end{aligned}
\end{equation}

Based on Lemma \ref{lemma2}, for any diagonal positive matrix $D > 0$ and $j \in \mathcal{M}$, we have
\begin{equation}\label{eq_important}
Re\{ \sigma [ (\sum_{i=1}^{N} \gamma_i  [A]^m_{s_i})  D]_j \} >0.
\end{equation}

Without loss of generality, we only consider the proof for the $m_{th}$-order matrices. So, we omit the
subscript $j$ in the following discussions.

According to structure of $K$, as before, we redefine the sub-indices of the $n$ columns of $A$ as follows: 
$$A=[\boldsymbol{a}_{1,1}, \cdots, \boldsymbol{a}_{1,p_1}, \boldsymbol{a}_{2,1}, \cdots, \boldsymbol{a}_{2,p_2}, \cdots, \cdots, \boldsymbol{a}_{m,1} \cdots,\boldsymbol{a}_{m,p_m}].$$ 

Correspondingly, we redefine the sub-indices of $\gamma_i>0$, $i \in \{1, \cdots, N\}$ in inequality (\ref{al_1}) as follows;
\begin{equation}
\gamma_{(\zeta_1, \zeta_2, \cdots, \zeta_m)},    
\end{equation}
where $\zeta_i \in \{1,2,\cdots,p_i\}$. The coefficient $\gamma_{(\zeta_1, \zeta_2, \cdots, \zeta_m)}$ means from each of the following group of vectors to select one vector to correspond to a term $\gamma_{(\zeta_1, \zeta_2, \cdots, \zeta_m)} [A]^m_{s_i}$ in inequality  (\ref{eq_important}):
\begin{equation}
\begin{aligned}
 Group\,1: & \,\,  \boldsymbol{a}_{1,1}, \cdots, \boldsymbol{a}_{1,p_1},\\ 
 Group\,2:  & \,\, \boldsymbol{a}_{2,1}, \cdots, \boldsymbol{a}_{2,p_2}, \\
  &\cdots, \cdots, \\
 Group\,m: & \,\, \boldsymbol{a}_{m,1} \cdots,\boldsymbol{a}_{m,p_m}. \\   
\end{aligned}
\end{equation}

Further assume $D=diag\{[d_1,d_2,\cdots, d_m]\}$.  Then the core part of the equation (\ref{eq_important}) can be expressed as follows:
\begin{equation} \label{carefu}
\begin{aligned}
     (\sum_{i=1}^{N} \gamma_i  [A]^m_{s_i})  D  
=[& d_1 ( \sum_{j=1}^{p_1} (\sum_{\zeta_2=1}^{p_2} \sum_{\zeta_3=1}^{p_3} \cdots \sum_{\zeta_m}^{p_m} \gamma_{( j, \zeta_2, \cdots, \zeta_m )} ) \boldsymbol{a}_{1,j} ), \\
     & d_2 ( \sum_{j=1}^{p_2} (\sum_{\zeta_1=1}^{p_1} \sum_{\zeta_3=1}^{p_3} \cdots \sum_{\zeta_m}^{p_m} \gamma_{( \zeta_1, j, \zeta_3, \cdots, \zeta_m )} ) \boldsymbol{a}_{2,j} ), \\
     & \cdots \cdots \cdots \\
      & d_m  ( \sum_{j=1}^{p_m} (\sum_{\zeta_1=1}^{p_1} \sum_{\zeta_2=1}^{p_2} \cdots \sum_{\zeta_{(m-1)}}^{p_{(m-1)}} \gamma_{( \zeta_1, \zeta_2, \cdots, \zeta_{(m-1)},j)} ) \boldsymbol{a}_{m,j} ) ].\\  
\end{aligned}
\end{equation}

On the other hand, the matrix $AEK$ can be expressed as 
\begin{equation}\label{eq_li_comb1}
\begin{aligned}
AEK &= [ \sum_{j=1}^{p_1} \varepsilon_{1j} k_{1j} \boldsymbol{a}_{1,j}, \sum_{j=1}^{p_2} \varepsilon_{2j} k_{2j} \boldsymbol{a}_{2,j}, \cdots, \sum_{j=1}^{p_m} \varepsilon_{mj} k_{mj} \boldsymbol{a}_{m,j}]\\
& = [ \sum_{j=1}^{p_1} \tilde {k}_{1j} \boldsymbol{a}_{1,j}, \sum_{j=1}^{p_2} \tilde {k}_{2j} \boldsymbol{a}_{2,j}, \cdots, \sum_{j=1}^{p_m} \tilde {k}_{mj} \boldsymbol{a}_{m,j}]\\
& = [ \tilde {k}_{11} \sum_{j=1}^{p_1} \frac{\tilde {k}_{1j}}{\tilde {k}_{11}}  \boldsymbol{a}_{1,j}, \tilde {k}_{21} \sum_{j=1}^{p_2} \frac{\tilde {k}_{2j}}{\tilde {k}_{21}}  \boldsymbol{a}_{2,j}, \cdots, \tilde {k}_{m1} \sum_{j=1}^{p_m} \frac{\tilde {k}_{mj}}{\tilde {k}_{m1}} \boldsymbol{a}_{m,j}]\\
& = [ \tilde {k}_{11} \sum_{j=1}^{p_1} \bar {k}_{1j} \boldsymbol{a}_{1,j}, \tilde {k}_{21} \sum_{j=1}^{p_2} \bar {k}_{2j} \boldsymbol{a}_{2,j}, \cdots, \tilde {k}_{m1} \sum_{j=1}^{p_m} \bar {k}_{mj} \boldsymbol{a}_{m,j}].\\
\end{aligned}
\end{equation}
In the above equation, $\tilde {k}_{ij} = \varepsilon_{i,j} k_{i,j} \ge 0$, and  $\bar {k}_{ij} = \frac{\tilde {k}_{i,j}}{\tilde {k}_{i,1}}$ (considering the case $\tilde {k}_{i,1} \ne 0$), which implies $\bar {k}_{i1}=1$.

Now, according to Lemma \ref{pailiezuhe_Lemma}, in Equation (\ref{carefu}), if we let $d_i = \tilde{k}_{i1}$ ($i \in \{1,2,\cdots,m\}$), and
\begin{equation}\label{sub_s}
\begin{cases}
\begin{aligned}
&\gamma_{(1, 1, \cdots, 1)} = 1 \\
&\gamma_{( \zeta_1, \zeta_2, \cdots, \zeta_{(i-1)}, j, \zeta_{(i+1)}, \cdots, \zeta_{m})} = \gamma_{( \zeta_1, \zeta_2, \cdots, \zeta_{(i-1)}, 1, \zeta_{(i+1)}, \cdots, \zeta_{m})} \bar{k}_{ij},
\end{aligned}
\end{cases}
\end{equation}
the left-hand side of Equation (\ref{carefu}) equals that of Equation (\ref{eq_li_comb1}). Thus, we can conclude, that for any non-negative diagonal matrices $E \in  \mathbb{R}^{n \times n}$, the inequality (\ref{suf_con_1}) is always satisfied.
\end{proof}

%


\begin{note}
Lemma \ref{simul_lemma22a} confirmed that a non-square simultaneously Volterra-Lyapunov stable matrix is also a non-square D-stable-like matrix. Simultaneously Volterra-Lyapunov stable matrix implies that there exists a common diagonal Lyapunov function for all its $m$-order squared matrices. This condition is relatively strong, which can even guarantee the instantaneous switching stability among all these squared matrices.    
\end{note}

\begin{note}
The simultaneous Volterra-Lyapunov stability condition in Lemma \ref{simul_lemma22a} is a strong condition especially for the case when the total number of non-squared matrices, $N$, is big. The following theorem relaxed the simultaneous Volterra-Lyapunov stability to individual Volterra-Lyapunov stability. 
\end{note}

\begin{lemma} \label{lem2} \cite{Cross:-Thre}
For normal matrices and in the class $\mathbb{F} = \{A : a_{ij} \ge 0 \,\,for \,\,all \,\, i \ne j \}$, D-stability and Volterra-Lyapunoc stability are all equivalent to stability. 
\end{lemma}

\begin{lemma} \label{lem33} 
For a set of normal matrices which are also in the class $\mathbb{F} = \{A : a_{ij} \ge 0 \,\,for \,\,all \,\, i \ne j \}$, D-stability, individual Volterra-Lyapunoc stability, and simultaneous Volterra-Lyapunoc stability are all equivalent to stability. 
\end{lemma}
\begin{proof}
Based on Lemma \ref{lem2}, this lemma can be easily proved.   
\end{proof}

\begin{theorem}  \label{mainreslt}
If any of the $m_{th}$-order squared matrices of  $A \in \mathbb{R}^{m \times n}$ is a normal matrix and in the class  $\mathbb{F}$, and also D-stable matrix,
then there exists a block diagonal non-square matrix $K \in \mathbb{C}^{n \times m}$ such that for all non-negative diagonal matrices $E \in  \mathbb{R}^{n \times n}$, and $j \in \mathcal{M}$, where $\mathcal{M}$ is the index set consisting of $k$ tuples of integers in the
range $1,\cdots, m$,
\begin{equation}\label{suf_con}
Re\{ \sigma_i ( [A E K]_j ) \} >0,
\end{equation}
where $\text{Re} \{ \sigma_i (M) \}$ represents the real part of the $i$-th eigenvalue of matrix $M$.
\end{theorem}
\begin{proof}
Based on Lemma \ref{lem33}, both individual and simultaneous Volterra-Lyapunoc stability of the matrix $A$ can be ensured. Then,  Lemma \ref{simul_lemma22a} can directly prove this theorem.
\end{proof}

\begin{conjecture}
If a real nonsquare matrix $A \in \mathbb{R}^{m \times n}$ is an individual Volterra-Lyapunov stable matrix (see Definition \ref{def_NSQ_lyp_ind}),
then, there exists a 
 block diagonal non-square matrix $K \in \mathbb{C}^{n \times m}$ such that for all non-negative diagonal matrices $E \in  \mathbb{R}^{n \times n}$, and $j \in \mathcal{M}$, where $\mathcal{M}$ is the index set consisting of $k$ tuples of integers in the
range $1,\cdots, m$,
\begin{equation}\label{suf_con}
Re\{ \sigma_i ( [A E K]_j ) \} >0,
\end{equation}
where $\text{Re} \{ \sigma_i (M) \}$ represents the real part of the $i$-th eigenvalue of matrix $M$.
\end{conjecture}

\begin{note}
If add extra conditions, the above conjecture can be proved. For example, if the limited number of redundant inputs is only in one channel. Then, the proof becomes straightforward. However, generally, the proof of this theorem is complex.  
\end{note}

\section{CONCLUSION}
In summary, we extended the concepts of key stable square matrices, such as D-stable and Volterra-Lyapunov stable matrices, to non-square cases. This was motivated by distributed control problems like decentralized stability and controllability in non-square processes. We established links between these square and non-square matrices and proposed conjectures for further research.

\bibliography{VO2}

\end{document}